\documentclass[oneside,11pt]{amsart}

\usepackage{amscd,amsmath,latexsym,amssymb}
\usepackage[mathcal]{euscript}
\usepackage[english]{babel}
\usepackage{amsmath} 
\usepackage{amsfonts}
\usepackage{amssymb} 
\usepackage{amsthm} 
\usepackage{amscd} 
\usepackage{pxfonts} 
\usepackage{textcomp}

\input xy
\xyoption{all}

\newtheorem{thm}{Theorem}[section]
\newtheorem{cor}[thm]{Corollary}

\newtheorem{lem}[thm]{Lemma}
\newtheorem{notation}[thm]{Notation}
\newtheorem{prop}[thm]{Proposition}
\theoremstyle{definition}

\theoremstyle{definition}

\theoremstyle{remark}
\newtheorem*{rem}{Remark}

\newcommand{\C}{{\mathbb C}}
\newcommand{\Z}{{\mathbb Z}}

\newcommand{\IR}{\mathbb{R}}
\newcommand{\IQ}{\mathbb{Q}}
\newcommand{\IZ}{\mathbb{Z}}
\newcommand{\IN}{\mathbb{N}}

\begin{document}

\title{Constructing free actions of p-groups on products of spheres.}

\author[Michele Klaus.]
{Michele Klaus.}
\address{Department of Mathematics
University of British Columbia, Vancouver BC V6T 1Z2, Canada}
\email{michele@math.ubc.ca}

\maketitle

\begin{abstract}
We prove that, for p an odd prime, every finite p-group
of rank 3 acts freely on a finite complex $X$ homotopy equivalent
to a product of three spheres.
\end{abstract}

\tableofcontents

\date{\today}

\section{Introduction.}

The origin of the study of group actions on spheres goes back
to Hopf and the spherical space form problem,
which asked for a classification of finite groups that can act freely on some sphere.
The first result was due to P.A. Smith \cite{smith}, who
showed that if a finite group $G$ acts freely on a sphere, then it
has periodic cohomology. Later Milnor \cite{milnor} gave a second necessary condition:
If a finite group $G$ acts freely on a sphere, then any element in $G$
of order 2 must be central. Finally Thomas, Wall and Madsen \cite{wall} were able to prove
the these two necessary conditions were in fact sufficient:
a finite group $G$ acts freely on a sphere if and only if it has
periodic cohomology and all involutions are central.

From the homotopy point of view, Swan \cite{swan} proved that a
finite group has periodic cohomology if and only if it acts freely
on a finite dimensional CW-complex homotopy equivalent to a sphere.
It is a classical result that a finite group has periodic cohomology
if and only if all of its abelian p-subgroups are cyclic. Based on
that and on their own algebraic results, Benson and Carlson
\cite{ben1} suggested the rank conjecture: for any finite group $G$
we have that $rk(G)=hrk(G)$, where $hrk(G)=min\left\{k\in\IN| \,\
\textit{G acts freely on a finite dimensional CW-complex} \,\
X\simeq S^{n_1}\times...\times S^{n_k} \right\}$ is the homotopy
rank of $G$. With this notation, Swan's result says that $rk(G)=1$
if and only if $hrk(G)=1$. In the same period Heller \cite{heller}
showed that $(\IZ_p)^3$ cannot act freely on a finite dimensional
CW-complex homotopy equivalent to a product of two spheres. More
recently Adem and Smith \cite{adem1} showed that if $rk(G)=2$ then
$hrk(G)=2$ for $G$ a p-group or $G$ a simple group (different from
$PSL_3(\mathbb{F}_p)$).
Adem \cite{adem3} proved also that every odd order rank two group
acts freely an a finite CW-complex $X\simeq S^n\times S^m$. This
also follows from a more general result of Jackson \cite{jackson2}.
Our main theorem is the following:

\begin{thm} \label{mainthm}
For p an odd prime, every finite p-group of rank 3
acts freely on a finite CW-complex homotopy equivalent
to the product of three spheres.
\end{thm}

Note that a converse to theorem \ref{mainthm} is given by Hanke in
\cite{hanke} in the sense that: if $(\IZ/p)^r$ acts freely on $X=
S^{n_1}\times ... \times S^{n_k}$ and if $p>3dim(X)$, then $r\leq
k$.
 We outline now the structure of the paper. Let p be an odd prime,
 let $G$ be a p-group and $S(V)$ a complex representation $G$-sphere.
In section 2 we first prove that, for all integer $k\geq 0$, there
exists a positive integer $q$ such that the group
$\pi_k(Aut_G(S(V^{\oplus q})))$ is finite. We then incorporate this
result in an outline of a known construction (\cite{adem1},
\cite{connolly}, \cite{Unlu}) that, in favorable conditions, gives a
strategy to build group actions on products of spheres with
controlled isotropy subgroups.

In section 3 we apply section 2 to prove that for $G$
a rank 3 p-group, p odd, there is a free
 finite $G$-CW-complex $X\simeq S^m\times S^d\times S^k$.
In section 4 we use section 2 to generalize theorem 3.2 in
\cite{adem1} for a p-group $G$: if $X$ is a finite dimensional
$G$-CW-complex with abelian isotropy, we show that there is
a free finite dimensional $G$-CW-complex $Y\simeq X\times S^1\times ...\times S^{n_k}$.
As a corollary we will be able to
build free finite
$G$-CW-complexes $X\simeq S^{n_1}\times ... \times S^{n_{rk(G)}}$
for $G$ a central extension of abelian p-groups.
Our results overlap here with those of \H{U}nl\H{u} and Yal\c{c}in \cite{UY}.

In section 5 we discuss the rank conjecture for some infinite
groups. The motivation comes from a result of Connolly and Prassidis
\cite{connolly} stating that: a group with finite virtual
cohomological dimension, which is countable and with rank 1 finite
subgroups, acts freely on a finite dimensional CW-complex $X\simeq
S^m$. We show that an effective $\Gamma$-sphere does not need to
exist but that the algebraic analogue still holds. More in detail:
First, we define an amalgamated product $\Gamma$ such that every
finite dimensional $\Gamma$-space homotopy equivalent to a sphere
has an isotropy subgroup of rank 2. Secondly, we prove that for all
groups $\Gamma$ with finite virtual cohomological dimension, there
is a finite dimensional $\IZ \left[\Gamma\right]$-projective complex
$\textbf{D}$ with $H^{\ast}(\textbf{D})\cong
    H^{\ast}(S^{n_1}\times ...\times S^{n_{rk(\Gamma)}})$.

\section{The general construction.}

The main result of this section is the construction of proposition
\ref{constr}. A key ingredient of the construction is proposition
\ref{tec}, which says that under some conditions $\pi_k(Aut_G(S^n))$
is finite. We begin with a series of lemmas and corollaries that we
assemble into a proof of proposition \ref{tec}. Lemmas \ref{lemma0},
\ref{lemma1} and \ref{lemma2} are individual results needed in the
proof of proposition \ref{tec}. Lemma \ref{pair} serves the proof of
lemma \ref{lemma2}.

\begin{lem} \label{lemma0}
Let $X$ be a $G$-CW-complex and let $Aut_G(X)$ be
 the monoid of $G$-equivariant self-homotopy equivalences of $X$.
For $k>0$, the map of unbased homotopy classes
$\varphi: \left[S^k, Aut_G(X)\right]\rightarrow \left[S^k\times X , X\right]_G$
is injective and, up to the choice of a connected component, factors through:
$$\xymatrix{ \left[S^k, Aut_G(X)\right] \ar[d] \ar[r]^{\varphi} & \left[S^k\times X , X\right]_G \\
    \pi_k(Aut_G(X)) \ar[ur] & }$$
In particular all $G$-equivariant homotopies $H:I\times S^k\times X\rightarrow X$
between maps in $Im(\varphi)$
can be taken to satisfy $H(t,\star,x)=H(t',\star,x)$ for
all $t, t'\in I$ and $x\in X$.
\end{lem}

\begin{proof}
The map $\varphi: \left[S^k, Aut_G(X)\right]\rightarrow \left[S^k\times X , X\right]_G$
is clearly well defined. To see that it is injective,
consider a $G$-equivariant homotopy $H:I\times S^k\times X  \rightarrow X$
from $\varphi (f)$ to $\varphi (g)$.
Clearly $H\arrowvert_{\left\{0\right\}\times\left\{x_0\right\}\times X}=\varphi (f)(x_0) \in Aut_G(X)$.
Which implies that $H\arrowvert_{\left\{t\right\}\times\left\{x\right\}\times X}\in Aut_G(X)$
for all $(t,x)\in I\times S^k$ because
$H\arrowvert_{\left\{t\right\}\times\left\{x\right\}\times X}\simeq
    H\arrowvert_{\left\{0\right\}\times\left\{x_0\right\}\times X}$
via a path in $I\times S^k$ from $(0,x_0)$ to $(t,x)$.
As a result, $H$ defines an homotopy from $f$ to $g$.
To prove that $\varphi$ factors, up to the choice of a connected component, through:
$$\xymatrix{ \left[S^k, Aut_G(X)\right] \ar[d] \ar[r]^{\varphi} & \left[S^k\times X , X\right]_G \\
    \pi_k(Aut_G(X)) \ar[ur] & }$$
we want to show that the map $\pi_k(Aut_G(X))\rightarrow \left[S^k,
Aut_G(X)\right]$ is a bijection. Observe that $Aut_G(X)$ is a
monoid, thus an $H$-space so that $\pi_1(Aut_G(X),Id)$ acts
trivially on $\pi_k(Aut_G(X),Id)$. The monoid $Aut_G(X)$ is very
nice because all of its connected components are homotopy equivalent
through maps of the form: $(Aut_G(X),Id)\rightarrow (Aut_G(X),f)$
with $g\mapsto f\circ g$. Consequently $\pi_1(Aut_G(X),f)$ acts
trivially on $\pi_k(Aut_G(X),f)$ for all $f\in Aut_G(X)$. We
conclude that $\pi_k(Aut_G(X))\rightarrow \left[S^k,
Aut_G(X)\right]$ is a bijection. The last claim directly follows
from the diagram.
\end{proof}

\begin{lem} \label{lemma1}
Let $G$ be a finite group acting on a space $X$.
Let $H_1<G$ be an isotropy subgroup maximal among
isotropy subgroups. Set
$X_1=\left\{x\in X| \,\ G_x\in (H_1)\right\}$,
where $(H_1)$ denotes the conjugacy class of $H_1$. We then have
that $Aut_G(X_1)\cong Aut_{WH_1}(X^{H_1})$
 (here $WH_1=NH_1/H_1$ is the Weil group).
\end{lem}

\begin{proof}
Let's begin by studying $X_1$. Clearly $X_1\subset\cup_{H\in (H_1)} X^H$.
Since $H\in (H_1)$ is supposed to be maximal,
we must have that if $x\in X^H$,
then $G_x=H$ so that $X_1=\cup_{H\in (H_1)} X^H$.
Similarly, if $x\in X^{H}\cap X^{H'}$,
then $H=G_x=H'$. As a result $X_1=\coprod_{H\in (H_1)} X^H$.

Observe next that a $G$-equivariant map
$f:X_1\rightarrow X_1$ restricts to a
$WH_1$-equivariant map $f_1:X^{H_1}\rightarrow X^{H_1}$
because $WH_1=NH_1/H_1$
and $H_1$ acts trivially on $X_1$. The same holds
for a $G$-equivariant homotopy $F:I\times X_1\rightarrow X_1$,
so that we have a well defined map
$res:Aut_G(X_1)\rightarrow Aut_{WH_1}(X^{H_1})$.

We want to show now that the map
$res:Aut_G(X_1)\rightarrow Aut_{WH_1}(X^{H_1})$ has an inverse
given by $res^{-1}(f)(x)=gf(g^{-1}x)$, where $g\in G$
is such that $g^{-1}x\in X^{H_1}$.
We begin by showing that $res^{-1}$ is well defined.
For all $x\in X_1=\coprod_{H\in (H_1)} X^H$ there is
a $g\in G$ such that $x\in X^{gH_1g^{-1}}$, so that
$g^{-1}x\in X^{H_1}$. Assume that
$\gamma\in G$ is also such that $\gamma^{-1}x\in X^{H_1}$.
Clearly $gH_1g^{-1}=H=\gamma H_1\gamma^{-1}$, where
$x\in X^H$. Thus $\gamma^{-1}gH_1g^{-1}\gamma=H_1$
so that $\gamma^{-1}g\in NH_1$. For $f\in Aut_{WH_1}(X^{H_1})$
we then have
$gf(g^{-1}x)=gg^{-1}\gamma f(\gamma^{-1}gg^{-1}x)=\gamma f(\gamma^{-1}x)$
because $f$ is $NH_1$-equivariant. Therefore $res^{-1}$ is
well defined.

Next, we show that $res^{-1}(f)$ is $G$-equivariant:
For $x\in X_1$, let again $g\in G$
be such that $g^{-1}x\in X^{H_1}$. For all
$g_0\in G$ we have that $(g_0g)^{-1}g_0x\in X^{H_1}$.
As a result $res^{-1}(f)(g_0x)=g_0gf((g_0g)^{-1}g_0x)=
g_0gf(g^{-1}x)=g_0res^{-1}(f)(x)$.
In the same way one can check that
if $f^{-1}$ is the homotopy inverse of
$f$ via homotopies $H$ and $H^{-1}$, then
$res^{-1}f^{-1}$ is the homotopy inverse of
$res^{-1}f$ via homotopies $res^{-1}H$ and $res^{-1}H^{-1}$.
Finally we observe that $res\circ res^{-1}=Id$
by choosing $g=1$, while $res^{-1}\circ res=Id$
because $res^{-1}f=f$ when $f$ is $G$-equivariant.
\end{proof}

\begin{lem} \label{pair}
Let $G$ be a finite group
and $S^n$ a linear $G$-sphere. If $0<k<n$
then $H^n(S^k\times S^n/G,\left\{\ast\right\}\times S^n/G,\Z)$ is finite.
\end{lem}

\begin{proof}
Consider the long exact sequence of the pair
$(S^k\times S^n/G,\left\{\ast\right\}\times S^n/G)$ with integer coefficients:
$$\xymatrix{  H^{n-1}(S^k\times S^n/G)\ar@{->>}[r] &
        H^{n-1}(\left\{\ast\right\}\times S^n/G) \ar[d] & & \\
    & H^n(S^k\times S^n/G,\left\{\ast\right\}\times S^n/G) \ar[r]  &
        H^{n}(S^k\times S^n/G)\ar[r]^{i^{\ast}} &
        H^{n}(\left\{\ast\right\}\times S^n/G)}$$
Clearly $H^n(S^k\times S^n/G,\left\{\ast\right\}\times S^n/G)\subset ker(i^{\ast})$.
But $H^{n}(S^k\times S^n/G)=H^{n}(\left\{\ast\right\}\times S^n/G)\oplus
    H^{n-k}(\left\{\ast\right\}\times S^n/G)$.
Thus for $i^{\ast}:H^{n}(S^k\times S^n/G)\rightarrow H^{n}(\left\{\ast\right\}\times S^n/G)$
we have that $Ker(i^{\ast})=H^{n-k}(\left\{\ast\right\}\times S^n/G)$.
Finally, the groups $H^{n-k}(S^n/G)$ are finite for $0<k<n$
because $H^{n-k}(S^n/G,\IQ )=0$ by the Vietoris-Begle theorem.
\end{proof}

\begin{lem} \label{lemma2}
Let $G$ be a finite group
and $S(V)$ a linear $G$-sphere.
For $H<G$ write $n_r(H)$ for the integer
such that $S(V^{\oplus r})^H=S^{n_r(H)}$. For all
$k>0$ there is an integer $q>0$ such that the groups:
$$H^{n_q(H_i)}(S^k\times S^{n_q(H_i)}/WH_i,
    \cup_{H>H_i}S^k\times S^{n_q(H)}/WH_i\cup \left\{\star\right\}\times S^{n_q(H_i)}/WH_i,\Z)$$
are finite for all $H_i$ with $n_1(H_i)>0$.
\end{lem}

\begin{proof}
Fix a subgroup $H_i<G$ such that $n_1(H_i)>0$.
If there is $H>H_i$ with $n_1(H)=n_1(H_i)$,
then the required cohomology group is zero
(it is of the form $H^{n(H_i)}(X,X,\Z)$).
Assume that for all $H>H_i$ we have $n_1(H)<n_1(H_i)$.
In this case we want so show that we can take
enough direct sums to be in the situation of
corollary \ref{pair}.

Let $n_{r,i}=max_{H>H_i}\left\{n_r(H)\right\}$ and $m_{r,i}=n_r(H_i)-n_{r,i}>0$.
Observe that $n_r(H)=rn_1(H)+(r-1)$
so that $n_{r,i}=rn_{1,i}+(r-1)$ and
$m_{r,i}=n_r(H_i)-n_{r,i}=rn_1(H_i)+(r-1)-(rn_{1,i}+(r-1))=rm_{1,i}$.
Therefore there is a $q_i$ big enough such that
$m_{r,i}>k+2$. In other words
$n_r(H_i)-k-2>n_{r,i}$.
We have found an integer $q_i>0$ such that
all the cells $\tau$ of the CW-complex
$ S^{n_{q_i}(H_i)}$ of dimension $dim(\tau )\geq n_{q_i}(H_i)-k-2$, are also cells
of the relative CW-complex
$(S^{n_{q_i}(H_i)}, \cup_{H>H_i} S^{n_{q_i}(H)})$.

We turn now our attention to the announced cohomology group.
By our condition on the cells of $S^{n_{q_i}(H_i)}$, we have that
the cells $\tau$ of the CW-complex
$ S^k\times S^{n_{q_i}(H_i)}/WH_i$ of dimension
$dim(\tau )\geq n_{q_i}(H_i)-2$, are also cells
of the relative CW-complex
$(S^k\times S^{n_{q_i}(H_i)}/WH_i, \cup_{H>H_i} S^k\times S^{n_{q_i}(H)}/WH_i)$.
Henceforth:
$H^{n_{q_i}(H_i)}(S^k\times S^{n_{q_i}(H_i)}/WH_i,
    \cup_{H>H_i}S^k\times S^{n_{q_i}(H)}/WH_i\cup \left\{\star\right\}\times S^{n_{q_i}(H_i)}, \Z)=
    H^{n(H_i)}(S^k\times S^{n_{q_1}(H_i)}/WH_i, \left\{\star\right\}\times S^{n_{q_i}(H_i)}/WH_i,\IZ)$.
This last group is finite, by virtue of lemma \ref{pair}.
We conclude by observing that we can then set $q=max_{H_i<G}\left\{q_i\right\}$.
\end{proof}

\begin{prop} \label{tec}
Let $G$ be a finite $p$-group. Let $S(V)$ be a complex
representation $G$-sphere. For all integer $k\geq 0$ there
exists an integer $q> 0$ such that
$\pi_k(Aut_G(S(V^{\oplus q})))$ is finite.
\end{prop}

\begin{proof}
If $k=0$, the result has been proven in \cite{ferrario}.
Assume that $k>0$. Before explaining how the proof proceeds,
we set up some notation:
Choose an ordering of the conjugacy classes of isotropy subgroups
$\left\{(H_1),...,(H_m)\right\}$ such that if $(H_j)<(H_i)$ then $i<j$.
Consider the filtration $S(V)_1\subset ...\subset S(V)_m=S(V)$ given
by $S(V)_i=\left\{x\in S(V)| \,\ (G_x)=(H_j); \,\ j\leq i\right\}$.
Observe that we have homomorphisms
$R_i:\pi_k(Aut_G(S(V)))\rightarrow \pi_k(Aut_G(S(V)_{i}))$
because $H(I\times S(V)_{i})\subset S(V)_{i}$ for all equivariant
$H:I\times S(V)\rightarrow S(V)$.
Similarly we have homomorphisms $S_i:\pi_k(Aut_G(S(V)_i))\rightarrow \pi_k(Aut_G(S(V)_{i-1}))$.
Here is how the proof runs. Look at the commutative diagram:
$$\xymatrix{ \pi_k(Aut_G(S(V))) \ar[r]^{R_{m}=Id} \ar[dr]^{R_{m-1}=S_{m}} \ar[dddr]^{R_i} \ar[dddddr]^{R_1} &
     \pi_k(Aut_G(S(V)_{m})) \ar[d]^{S_{m}} \\
    & \pi_k(Aut_G(S(V)_{m-1})) \ar[d]^{S_{m-1}} \\ & ... \ar[d]^{S_{i+1}} \\
     & \pi_k(Aut_G(S(V)_{i})) \ar[d]^{S_{i}} \\  & ... \ar[d]^{S_2} \\ & \pi_k(Aut_G(S(V)_{1}))}$$
Clearly to prove that $\pi_k(Aut_G(S(V)))$ is finite is the same as
to prove that $Im(R_m)$ is finite. To prove that $Im(R_m)$ is
finite, we will show by induction over $i$ that $Im(R_i)$ is finite.
Such an induction can be performed by showing that $Im(R_1)$ is
finite and that $S^{-1}_{i}(R_{i-1}(f))\cap Im(R_i)$ is finite for
all $i$ and for all $f\in \pi_k(Aut_G(S(V)))$. This outline can only
be carried out up to replacing $S(V)$ with some power $S(V^{\oplus
q})$.

\bigskip
We begin by showing that there is $q_1>0$ such that
$\pi_k(Aut_G(S(V^{\oplus q_1})_1))$ is finite.
In particular we will have that
$Im(R_1)\subset \pi_k(Aut_G(S(V^{\oplus q_1})_1))$ is finite.
For $H<G$ write $n_r(H)$ for the integer
such that $S(V^{\oplus r})^H=S^{n_r(H)}$.
Observe that $n_r(H)=rn(H)+(r-1)$.
By lemma \ref{lemma1} we have that
$\pi_k(Aut_G(S(V)_1))=\pi_k(Aut_{WH_1}(S^{n_1(H_1)}))$.
The $WH_1$-action on $S^{n_1(H_1)}$ is
free because $H_1$ is maximal among isotropy subgroups.
Therefore proposition 2.4 of \cite{connolly}
says that $\pi_k(Aut_{WH_1}(S^{n_1(H_1)}))$ is finite
if $k<n_1(H_1)-1$. If $k\geq n_1(H_1)-1$, then
there is a $q_1>0$ for which $k<q_1n_1(H_1)+(q_1-1)-1=n_{q_1}(H_1)-1$.
As a result $\pi_k(Aut_G(S(V^{\oplus q_1})_1))=\pi_k(Aut_{WH_1}(S^{n_{q_1}(H_1)}))$
is finite (always by proposition 2.4 of \cite{connolly}).

\bigskip
As explained above, the second and last step is to prove that there
is $q\geq q_1$ such that $S^{-1}_{i}(R_{i-1}(f))\cap Im(R_i)$ is
finite for all $i$ and for all $f\in \pi_k(Aut_G(S(V^{\oplus q})))$.
For that purpose we are going to use equivariant obstruction theory
a la Tom Dieck (see \cite{tomdieck1} section 8 and \cite{tomdieck2}
chapter 2). We begin with some preliminaries. As in lemma
\ref{lemma2}, let $q'>0$ be such that the groups:
$$H^{n_{q'}(H')}(S^k\times S^{n_{q'}(H')}/WH',
    \cup_{H>H'}S^k\times S^{n_{q'}(H)}/WH'\cup \left\{\star\right\}\times S^{n_{q'}(H')}/WH',\Z)$$
are finite for all $H'<G$ with $n_1(H')>0$. Let $q=max \left\{q_1,q'\right\}$.
To simplify the notation we write $W=V^{\oplus q}$,
$X=S^k\times S(W)$ and
$\bar{X}^{H_i}=\cup_{H>H_i}X^H\cup \left\{\star\right\}\times S(W)^{H_i}$.
With this notation we have that the group:
$$H^{n_{q}(H_i)}(X^{H_i}/WH_i, \bar{X}^{H_i}/WH_i, \pi_{n_{q}(H_i)}(S^{n_{q}(H_i)}))$$
is finite by lemma \ref{lemma2}, while if $r\neq n_{q}(H_i)$ then the groups:
$$H^{r}(X^{H_i}/WH_i, \bar{X}^{H_i}/WH_i, \pi_{r}(S^{n_{q}(H_i)}))$$
are finite because they are finitely generated torsion abelian
groups. (The fixed points of a complex representation spheres are
odd-dimensional spheres whose homotopy groups are all but one
finite).

\bigskip
To use equivariant obstruction theory we need one last observation:
By lemma \ref{lemma0} there is an injection
$\pi_k(Aut_G(S(W)_i))\rightarrow \left[S^k\times S(W)_i,S(W)_i\right]_G$
yielding a commutative diagram with injective columns:
$$\xymatrix{ \pi_k(Aut_G(S(W)_i)) \ar[r]^{S_i} \ar[d]_{\varphi_{i}} & \pi_k(Aut_G(S(W)_{i-1})) \ar[d]_{\varphi_{i-1}} \\
    \left[ S^k\times S(W)_i, S(W)_i\right]_G \ar[r]_{s_i} & \left[ S^k\times S(W)_{i-1}, S(W)_{i-1}\right]_G}$$
As a consequence, to prove that $S^{-1}_{i}(R_{i-1}(f))\cap Im(R_i)$
is finite, it is enough to prove that
$s^{-1}_{i}(\varphi_{i-1}(R_{i-1}(f)))\cap \varphi_i(Im(R_i))$ is
finite. By abuse of notation we will keep on writing $S_i$ and
$R_{i-1}(f)$, but we will think of them as living in the bottom row
of the diagram.

\bigskip
We are now in condition of applying equivariant obstruction theory
inductively over $r$ to each of the diagrams:
$$\xymatrix{ \left[ Sk_{r+1}(X_i,X_{i-1}), S(W)_i\right]_G \ar[r]^{S_{i,r+1}} \ar[d]_{Sk_{r}} &
    \left[ X_{i-1}, S(W)_{i-1}\right]_G \\
    \left[ Sk_{r}(X_i,X_{i-1}), S(W)_i\right]_G \ar[ur]_{S_{i,r}} & }$$
If $r=0$, then $Sk_0(X_i,X_{i-1})=X_{i-1}\coprod \left\{
x_0,...,x_l\right\}$. Consequently, $S^{-1}_{i,0}(R_{i-1}(f))\cap
Sk_{0}(Im(R_i))$ depends on the connected components of
$S(W)_{i-1}$. But $S(W)_{i-1}$ has finitely many connected
components because it is a finite CW-complex, therefore
$S^{-1}_{i,0}(R_{i-1}(f))\cap Sk_{0}(Im(R_i))$ is finite. From now
on, to simplify the notation, we are going to write $f_i=R_i(f)$ for
all possible $i$ and $f$. Assume that $S^{-1}_{i,r}(f_{i-1})\cap
Sk_{r}(Im(R_i))=\left\{g^1_{i,r},...,g^t_{i,r}\right\}$ is finite of
order $t$ (i.e. $g^j_{i,r}\neq g^l_{i,r}$ if $j\neq l$). For each
$g_{i,r+1}\in S^{-1}_{i,r+1}(f_{i-1})\cap Sk_{r+1}(Im(R_i))$ there
is a unique $g^j_{i,r}$ and a homotopy $h$ from
$g_{i,r}=g_{i}\arrowvert_{Sk_r(X_i,
    X_{i-1}\cup\left\{\star\right\}\times S(W))}$
to $g^j_{i,r}=g^j_i\arrowvert_{Sk_r(X_i,
    X_{i-1}\cup\left\{\star\right\}\times S(W))}$.

The crucial observation here is that the homotopy $h$
can be supposed to be constant over the subcomplex
$\left\{\star \right\}\times S(W)$: if $h$
extends to a homotopy between $g_i$ and $g^j_i$, then,
by lemma \ref{lemma0}, the homotopy $H:I\times S^k\times S(W)_i\rightarrow S(W)_i$
between $g_i$ and $g^j_i$
can be taken to satisfy $H(t,\star,y)=H(t',\star,y)$ for
all $t, t'\in I$ and $y\in S(W)_i$.

\bigskip
We write $d(g_{i,r},h,g^j_{i,r})\in
    H^{r+1}(X^{H_i}/WH_i, \bar{X}^{H_i}/WH_i, \pi_{r+1}((S^n)^{H_i}))$ for
 the difference cocycle as in \cite{tomdieck1} section 8 and \cite{tomdieck2} chapter 2
(see also \cite{steenrod}).
Notice that $d(g_{i,r},h,g^j_{i,r})$ is a cocycle because $g^j,g\in \pi_k (Aut_G(S(W)))$
(see lemma 3.14 in \cite{tomdieck2}).
The properties of the difference cocycle are given in
 3.13 of chapter 2 of \cite{tomdieck2}. In particular
we have that, if $d(g'_{i,r},h',g^j_{i,r})=d(g_{i,r},h,g^j_{i,r})$,
then $d(g'_{i,r},h'+h^{-1},g_{i,r})=
    d(g'_{i,r},h',g^j_{i,r})+d(g^j_{i,r},h^{-1},g_{i,r})=
    d(g'_{i,r},h',g^j_{i,r})-d(g_{i,r},h,g^j_{i,r})=0$, so that
$g_{i,r+1}\simeq g'_{i,r+1}$.
We can therefore define an injection:
$$d:S^{-1}_{i,r+1}(f_{i-1})\cap Sk_{r+1}(Im(R_i))\rightarrow \coprod_{j=1}^t
    \left\{(g^j_{i,r})\right\}\times H^{r+1}(X^{H_i}/WH_i, \bar{X}^{H_i}/WH_i, \pi_{r+1}(S^{n_q(H_i)}))$$
by setting $g_i\mapsto \left\{ g^j_i\right\}\times d(g_{i,r},h,g^j_{i,r})$.
Since we chose the integer $q$ in order to have
all the cohomology groups on the right hand side to be finite,
we must have that the left hand side is finite as well.

Summarizing, by induction we have that $S^{-1}_{i,r+1}(f_{i-1})\cap Sk_{r+1}(Im(R_i))$
is finite for all $r$. Since $X$ is finite dimensional,
this shows that $S^{-1}_{i}(f_{i-1})$ is finite.
We conclude as explained in the outline at the beginning of this proof.
\end{proof}

We introduce next the following notation:

\begin{notation}
Let $G$ be a finite group and $X$ a
$G$-CW-complex. We write:
$$rk_X(G)=max \left\{ n\in\IN |
    \,\ \text{there exists} \,\ G_{\sigma} \,\ \text{with} \,\ rk(G_{\sigma})=n \right\}$$
\end{notation}

We can now state the main result of the section:

\begin{prop} \label{constr}
 Let $G$ be a finite group and let $X$
be a finite dimensional $G$-CW-complex.
Assume that to each isotropy subgroup $G_{\sigma}$
we can associate a representation
$\rho_{\sigma}:G_{\sigma}\rightarrow U(n)$
such that $\rho_{\sigma}\arrowvert_{G\tau}\cong\rho_{\tau}$
whenever $G_{\tau}<G_{\sigma}$. If $\rho_{\sigma}$ is
fixed point free for all $G_{\sigma}$ with
$rk(G_{\sigma})=rk_X(G)$, then there
exists a finite dimensional $G$-CW-complex
$E\cong X\times S^m$ with $rk_E(G)=rk_X(G)-1$.
Moreover, if $X$ is finite then $E$ is finite as well.
\end{prop}

\begin{proof}
 The proof follows \cite{connolly}.
We refer the reader to \cite{Unlu} for the details.
Write
$S^{2n-1}_{\sigma}$ for the linear sphere associated to $\rho_{\sigma}$.
We want to glue these spheres into a $G$-equivariant spherical
fibration over $X$. We will proceed by induction over the skeleton
of $X$.
For every $G$-orbit of the 0-skeleton, choose a representant $\sigma$
and define a map $G\times_{G_{\sigma}} S^{2n-1}_{\sigma}\rightarrow
X^0$ by $(g,x)\mapsto g\cdot\sigma$. This defines a
$G$-equivariant spherical fibration $S^{2n-1}\rightarrow E_0\rightarrow
Sk_0(X)$ whose total space is a finite dimensional $G$-CW-complex.
Clearly if $\rho_{\sigma}$ is fixed point free
for all $G_{\sigma}$ with $rk(G_{\sigma})=rk_X(G)$,
then $rk(E_0)=rk_X(G)-1$.

The inductive step is next. Suppose given a $G$-equivariant spherical fibration
over the $(k-1)$-skeleton $\ast^{q_{k-1}}S^{2n-1}\rightarrow E_{k-1}\rightarrow Sk_{k-1}(X)$
whose total space is a finite dimensional $G$-CW-complex.
Assume also that if $\rho_{\sigma}$ is fixed point free
for all $G_{\sigma}$ with $rk(G_{\sigma})=rk_X(G)$,
then $rk(E_{k-1})=rk_X(G)-1$.
Now, for every $G$-orbit of a $k$-cell, choose a representative $\sigma$.
The $G_{\sigma}$-equivariant fibration
$\ast^{q_{k-1}}S^{2n-1}\rightarrow E_{k-1}\arrowvert_{\partial\sigma}\rightarrow \partial\sigma$ is classified by
an element $a_{\sigma}\in \pi_{k-2}(Aut_{G_{\sigma}}(\ast^{q_{k-1}}S^{2n-1}))$.

We want to have $a_{\sigma}=0$:
Observe that, in general, for two complex $G$-spheres $S(V)$ and $S(W)$,
we have that $S(V\oplus W)\cong S(V)\ast S(W)$ as $G$-spheres.
Therefore, by lemma \ref{tec}, we can take
enough Whitney sums of the fibration
$\ast^{q_{k-1}}S^{2n-1}\rightarrow E_{k-1}\rightarrow Sk_{k-1}(X)$
to guarantee that $a_{\sigma}=0$ (see
lemma 2.3 and proposition 2.4 in \cite{connolly}). We
can then extend the $G_{\sigma}$-equivariant fibration
$\ast^{q_{k}}S^{2n-1}\rightarrow E_{k-1}\arrowvert_{\partial\sigma}
\rightarrow\partial\sigma$ equivariantly across the cell $\sigma$. We define a
$G$-equivariant spherical fibration over the orbit of $\sigma$ by
$G\times_{G_{\sigma}}\ast^{q_k} S^{2n-1}_{\sigma} \rightarrow G\sigma$ with $(g,x)\mapsto g\cdot\sigma$.

Repeating the procedure for all the representatives of
the $G$-orbits of $k$-cells,
we recover a $G$-equivariant spherical fibration
$\ast^{q_{k}}S^{2n-1}\rightarrow E_k\rightarrow Sk_k(X)$
with total space a finite dimensional $G$-CW-complex.
Clearly if $\rho_{\sigma}$ is fixed point free
for all $G_{\sigma}$ with $rk(G_{\sigma})=rk_X(G)$,
then $rk(E_k)=rk_X(G)-1$.
We conclude noticing that,
by proposition 2.8 in \cite{adem1}, up to taking further fiber
joins, we can assume that the total fibration $\ast^{q}S^{2n-1}\rightarrow
E\rightarrow X$ is a product one.

For the last statement, one can observe that
all the constructions take place in the category
of finite CW-complexes, providing that the initial space $X$ is a finite CW-complex.
\end{proof}

\section{Rank 3 p-groups (p odd).}

The results of this section have also been announced by
Jackson in \cite{jackson}. We give a proof
which uses the group theory developed there.
For the convenience of the reader we reproduce it here.

\begin{lem} \label{michio}
If $G$ is a finite $p$-group with $rk(G)=3$ and $rk(Z(G))=1$,
then there exists a normal abelian subgroup $Q<G$ of type $(p,p)$
with $Q\cap Z(G)\neq 0$.
\end{lem}

\begin{proof}
This is proven by Suzuki in \cite{suzuki}, section 4.
\end{proof}

\begin{prop} \label{class}
Let $G$ be a finite $p$-group with $p>2$, $rk(G)=3$ and
$rk(Z(G))=1$. Let $Q$ be an abelian normal subgroup of type $(p,p)$
as above. Suppose that $H<G$ with $H\cap Z(G)=0$ and $|H|=p^n$. Then
either $H$ is cyclic, $H< C_G(Q)$, $H$ is abelian of type
$(p,p^{n-1})$ or $H\cong M(p^n)=<x,y|x^{p^{n-1}}=y^p=1,
y^{-1}xy=x^{1+p^{n-2}}>$.
\end{prop}

\begin{proof}
If $rk(H)=1$ then $H$ is cyclic since $p>2$. Suppose that $rk(H)=2$
and $H\cap Q\neq 0$. By assumption $Z(G)\cap Q=\Z/p$, $H\cap Z(G)=0$
and $Q\cap H=\Z/p$. The map $c: H\rightarrow Aut(H\cap Q)$ given by
$c_h(x)=hxh^{-1}$ is well defined because $Q$ is normal (by lemma
\ref{michio}). Since $|H|=p^n$ and $|Aut(H\cap Q)|=p-1$, we have
that the map $c$ is trivial. As a result we have that $H<C_G(Q)$.

Assume now that $H\cap Q=0$. In this case $H\cap C_G(Q)\neq H$ since
otherwise we would have $rk(G)>3$. Set $K=H\cap C_G(Q)$ and observe
that $K$ is cyclic (else we would have $rk(G)>3$). Assume for a
moment that $\left[G:C_G(Q)\right]=p$. In this case
$\left[H:K\right]=p$, in other words $H$ has a maximal  cyclic
subgroup. By \cite{suzuki} section 4, $H$ needs then to be abelian
of type $(p,p^{n-1})$ or $M(p^n)$.

We still have to prove that $\left[G:C_G(Q)\right]=p$. The group $G$
acts on $Q$ by conjugation and for each element $q$ of $Q$ not in
center of $G$, we have that $G_q=C_G(Q)$. As a result
$|C_G(Q)|=|G_q|=|G|/p$ since $Q\cong (Z/p)^2$ with the first
coordinate in the center $Z(G)$.
\end{proof}

\begin{prop}  \label{classfct}
Let $G$ be a finite $p$-group with $p>2$, $rk(G)=3$ and
$rk(Z(G))=1$. There exists a class function $\beta: G \rightarrow
\C$ such that for any subgroup $H\subset G$, with $H\cap Z(G)=0$,
the restriction $\beta\arrowvert_{H}$ is a complex character of $H$.
If in addition $H$ is a rank two elementary abelian subgroup, then
the character $\beta\arrowvert_H$ corresponds to an isomorphism
class of fixed-point free representations.
\end{prop}

\begin{proof}
Define $\beta :G \rightarrow\C$ as follows:
$$x\mapsto \left\{
                      \begin{array}{ll}
                        (p^2-p)|G|, & \hbox{if $x=0$;} \\
                        0, & \hbox{if $x\in Z(G)\setminus 0$;} \\
                        -p|G|, & \hbox{if $x\in Q\setminus Z(G)$;} \\
                         0, & \hbox{if $x\in C_G(Q)\setminus Q$;} \\
                        -|G|, & \hbox{if $x\in G\setminus C_G(Q)$ of order p;}\\
        0 & \hbox{if $x\in G\setminus C_G(Q)$ of order greater than p.}
                      \end{array}
                    \right.$$
The map $\beta$ is a class function because we have the following
sequence of subgroups  each normal in $G$:
$$0<Z(G)<Q<C_G(Q)<G.$$ Consider first an elementary
abelian subgroup $H$ of $G$ of rank 2 and which intersects trivially
the center $Z(G)$. If $H\cap Q\neq 0$ then:
$$\beta\arrowvert_H =|G|\displaystyle\sum_{i=0}^{p-1}\displaystyle\sum_{j=1}^{p-1}\phi_i\phi_j$$
where $\phi_k$ is the $k$-th irreducible character of $\Z /p$. If
$H\cap Q=0$ then:
$$\beta\arrowvert_H =(p-1)|G|/p\left(\displaystyle\sum_{i=0}^{p-1}\displaystyle\sum_{j=1}^{p-1}\phi_i\phi_j\right)+
    |G|\displaystyle\sum_{i=1}^{p-1}\phi_i\phi_0.$$
Consider now a subgroup $H$ of $G$ with $H\cap Z(G)=0$. We will
proceed case by case using the classification above.
\begin{enumerate}
\item If $H\cap Q\neq 0$ then $H\subset C_G(Q)$
    and $|K|=p$ with $K=Q\cap H$. Let $\phi$ be
    the character of $K$ which is $p-1$ on the identity
    and $-1$ for each other element of $K$. Then:
    $$\beta\arrowvert_H =\frac{p|G|}{|H:K|}Ind_K^H \phi.$$

\item If $H\cap Q =0$ and $H\subset C_G(Q)$ then
    $H$ is cyclic and
    $\beta\arrowvert_H= |G|/|H|(p^2-p)\phi$ where
    $\phi$ is the character of $H$ that is $|H|$ on the
    identity and $0$ elsewhere.

\item If $H\cap Q=0$ and $H$ is cyclic with $H\cap C_G(Q)=0$,
    then $|H|=p$ and $\beta\arrowvert_H =|G|/p \phi$ where
    $\phi$ is $(p^3-p^2)$ on the identity and $-p$ elsewhere.

\item Assume that $H\cap Q=0$ and that $H$ is abelian of type
    $(p,p^{n-1})$. Write $H=<x,y\in H|x^p=y^{p^{n-1}}=1, \left[x,y\right]=1>$.
    Notice that $<y>=H\cap C_G(Q)$. For each $1\leq i\leq p-1$ set
    $H_i=<xy^{ip^{n-2}}>$. Clearly $|H_i|=p$, $H_i\cap C_G(Q)=0$
    and $H_i\cap H_j=0$ if $i\neq j$.

    Let $\phi_i$ be the character of $H_i$ which is $p-1$ on
    the identity and $-1$ elsewhere. Set:
    $$\phi=\displaystyle\sum_{i=1}^{p-1}Ind_{H_i}^H\phi_i.$$
    Since $\phi (1)=|H|(p-1)$, $\phi(z)=-|H|/p$  for $z\in H\setminus <y>$ and
    $\phi (z)=0$ for $z\in <y>$; we conclude that $\beta\arrowvert_H= p|G|/|H|\phi$.

\item If $H\cap Q=0$ and $H\cong M(p^n)$,
    we can write $H=<x,y|x^{p^{n-1}}=y^p=1, y^{-1}xy=x^{1+p^{n+2}}>$.
    Let $N=<x^{p^{n-2}},y>\cong (\Z/p)^2$ which is normal in $H$.
    Let $\phi$ be the character of $H$ which is $p-1$ on the identity and such that
    $\phi(x^i)=-1$ for all $1\leq i\leq p-1$ and $\phi(z) =0$ when $z\in N\setminus <y>$.
    Then $\beta\arrowvert_H =Ind_N^H p|G|/|H:N|\phi$.
\end{enumerate}
\end{proof}

We can now turn our attention to the topological problem:

\begin{prop}
For every odd order rank 3 p-group $G$, there is a finite
dimensional $G$-CW-complex $X\simeq S^m\times S^n$ with cyclic
isotropy subgroups.
\end{prop}

\begin{proof}
If $Z(G)$ is not cyclic, then it is enough to consider the linear
spheres of representations of $G$ induced from free
representations of some $\IZ /p\times\IZ /p <Z(G)$.

Assume that $Z(G)$ is cyclic and let $S^m$ be the linear $G$-sphere
obtained by inducing from a free linear action of $Z(G)$. The
isotropy subgroups for this action are the one described in
proposition \ref{class}. The conditions of proposition
\ref{constr} are fulfilled by proposition \ref{classfct}. The conclusion follows.
\end{proof}

As a direct consequence of theorem 3.2 in \cite{adem1} we obtain:

\begin{thm} \label{yetagain}
For every odd order rank 3 p-group $G$, there is a free finite
$G$-CW-complex $X\cong S^m\times S^n\times S^k$.
\end{thm}

Note that a converse to theorem \ref{yetagain} is given by Hanke in
\cite{hanke} in the sense that: if $(\IZ/p)^r$ acts freely on $X=
S^{n_1}\times ... \times S^{n_k}$ and if $p>3dim(X)$, then $r\leq
k$.

\begin{rem}
For $p=2$ the situation is more complicated
because of the classification of subgroups. A 2-group
of rank 1 can be either cyclic or generalized quaternion.
A 2-group with a maximal abelian subgroup can be cyclic,
generalized quaternion, dihedral, $M(2^n)$ (see proposition
\ref{class}) or $S_{4m}=<x,y|x^{2m}=y^2=1,
y^{-1}xy=x^{m-1}>$ (see \cite{suzuki} section 4 chapter 4).
For $p=2$, the class function of proposition \ref{classfct}
does not restrict to characters over
the subgroups, in general.
\end{rem}

\section{Abelian isotropy p-groups.}

We begin by generalizing theorem 3.2 in \cite{adem1}
for p-groups.

\begin{thm}
Let $G$ be a finite p-group and let $X$ be a
finite dimensional $G$-CW-complex with
$G_{\sigma}$ abelian for all cells $\sigma\subset X$. Then there is
a free finite dimensional $G$-CW-complex
$Y\simeq X\times S^{n_1}\times ... \times S^{rk_X(G)}$.
Moreover, if $X$ is finite, then $Y$ is finite as well.
\end{thm}

\begin{proof}
We prove the theorem by induction over $rk_X(G)$. If $rk_X(G)=1$,
the theorem has been proven by Adem and Smith (3.2 in \cite{adem1}).

The inductive step follows. By virtue of
proposition \ref{constr}, we only need to associate to
each isotropy subgroup $G_{\sigma}$
a representation
$\rho_{\sigma}:G_{\sigma} \rightarrow U(m)$
such that
$\rho_{\sigma}\arrowvert_{G\tau}\cong
\rho_{\tau}$
whenever $G_{\tau}<G_{\sigma}$ and such that $
\rho_{\sigma}$ is
fixed point free for all $G_{\sigma}$ with
$rk(G_{\sigma})=rk_X(G)$.

Consider the
class function  $\beta :G \rightarrow\C$ given by:
$$x\mapsto \left\{
                      \begin{array}{ll}
                        |G|(p^{rk_X(G)}-1), & \hbox{if $x=0$;} \\
                        -|G|, & \hbox{if $o(x)=p$;} \\
                        0, & \hbox{otherwise.}
                      \end{array}
                    \right.$$
To simplify the notation write $A=G_{\sigma}$
for an isotropy subgroup (which is abelian by
hypothesis). We need to prove that $\beta\arrowvert_A$
is a character which is fixed point free whenever
$A\cong (\Z/p)^{rk_X(G)}$. Set
$A_p=\left\{0\right\}\cup\left\{x\in A| o(x)=p\right\}$.
Since $A$ is abelian we have $A_p\triangleleft A$.
Fix an injection $f:A_p\rightarrow (\Z/p)^{rk_X(G)}$.
Write $
\rho_0:(\Z/p)^{rk_X(G)}\rightarrow U(p^{rk_X(G)}-1)$
for the reduced regular representation and let $
\rho=
\rho_0\circ f$
be the representation $A_p\rightarrow (\Z/p)^{rk_X(G)}\rightarrow U(p^{rk_X(G)}-1)$.
Consider finally the representation of $A$ given by
$\eta=|G||A_p|/|A|Ind_{A_p}^A
\rho$.
Clearly $\eta (0)=|G|(p^{rk_X(G)}-1)$,
$\eta (x)=-|G|$ if $x\in A_p\setminus 0$ while
$\eta (x)=0$ if $x\notin A_p$. As a result $\beta\arrowvert_A=\eta$.
Let now $A\cong (\Z/p)^{rk_X(G)}$. Clearly
$\beta\arrowvert_A$ is a multiple of the reduced
regular representation, thus fixed point free.
\end{proof}

\begin{cor}
Let $G$ be a p-group. Assume that
$G$ is a central extension of abelians, then there is a free finite
 $G$-CW-complex $X\simeq S^{n_1}\times ...\times S^{n_{rk(G)}}$.
The result in particular holds for extraspecial p-groups.
\end{cor}

\begin{proof}
Let $X=S^{n_1}\times ...\times S^{n_{rk(Z(G))}}$
be the product of the $G$-spheres arising from suitable
representations of the center. Clearly
$rk_X(G)=rk(G)-rk(Z(G))$ and $G_{\sigma}$
is abelian. The conclusion follows.
\end{proof}

\section{Infinite groups.}

As pointed out in \cite{connolly}, there is a class of infinite
groups which is worth considering, when studying the rank conjecture
mentioned in the introduction (which is usually stated for finite
groups). This is the class of groups $\Gamma$ of finite virtual
cohomological dimension. Recall that, by definition, a group
$\Gamma$ has finite virtual cohomological dimension, if it has a
finite index subgroup $\Gamma '<\Gamma$ with finite cohomological
dimension (that is to say: $H^n(\Gamma ',M)=0$ for all coefficients
$M$ and for all $n$ big enough). Writing $vcd$ for virtual
cohomological dimension and $cd$ for cohomological dimension, one
can show that the number $vcd(\Gamma)=cd(\Gamma ')$ is well defined.
See for example \cite{brown} for background on groups with finite
virtual cohomological dimension. The crucial property that makes
them interesting to us is the following: for any such group $\Gamma$
there exists a finite dimensional $\Gamma$-CW-complex
$\mathfrak{E}\Gamma$ with $|\Gamma_x|<\infty$ for all $x\in
\mathfrak{E}\Gamma$.

It is already known that a group with finite virtual cohomological dimension,
which is countable and with rank 1 finite
subgroups, acts freely on a finite dimensional CW-complex $X\simeq S^m$ \cite{connolly}.
The next step would be to prove the analogue result
for groups $\Gamma$ with rank 2 finite subgroups.
The easiest examples to consider are amalgamated products
$\Gamma=G_1\ast_{G_0} G_2$, where $G_i$ is a finite group
for $i=0,1,2$ and $G_0<G_i$ for $i=1,2$. In this case,
for every finite subgroup $H<\Gamma$, there is $\gamma\in\Gamma$
such that $\gamma H \gamma^{-1}<G_i$ for $i=1$ or $i=2$ (see \cite{serre}).
In particular $rk(\Gamma)=max\left\{rk(G_1),rk(G_2)\right\}$.
The first attempt would be to find an effective $\Gamma$-sphere,
i.e. a $\Gamma$-sphere with rank 1 isotropy subgroups.
In the first subsection
we exhibit an amalgamation of two p-groups
which doesn't have an effective $\Gamma$-sphere.

Recall from \cite{ben1} that, for a finite group $G$, we have that $rk(G)=r$ if and only if
there are $r$ finite dimensional
$\IZ \left[ G \right]$-complexes $\textbf{C}_1,...,\textbf{C}_{r}$
such that $\textbf{C}=\textbf{C}_1\otimes...\otimes\textbf{C}_{r}$
is a complex of projective $\IZ \left[ G \right]$-modules with
$H^{\ast}(\textbf{C})\cong
    H^{\ast}(S^{n_1}\times ...\times S^{n_{r}})$.
In the second subsection we prove a similar result: for every group $\Gamma$
with $vcd(\Gamma )<\infty$, there is a finite dimensional,
contractible $\IZ \left[\Gamma\right]$-complex $\textbf{C}$
and $rk(\Gamma)$ finite dimensional
$\IZ \left[\Gamma\right]$-complexes $\textbf{C}_1,...,\textbf{C}_{rk(\Gamma)}$
such that $\textbf{D}=\textbf{C}\otimes\textbf{C}_1\otimes...\otimes\textbf{C}_{rk(\Gamma)}$
is a complex of projective $\IZ \left[\Gamma\right]$-modules with
$H^{\ast}(\textbf{D})\cong
    H^{\ast}(S^{n_1}\times ...\times S^{n_{rk(\Gamma)}})$. As a result,
the group $\Gamma$ introduced in the first subsection here below satisfies the
algebraic analogue of the rank conjecture but doesn't have an effective
$\Gamma$-sphere. The geometric problem of knowing
whether or not $\Gamma$ acts freely on a product
of two spheres is still open.

\subsection{A group without effective action on a sphere.}

Let $E$ and $E'$ be two copies of the extraspecial $p$-group of
order $p^3$ and exponent $p$. (Such a group can be identified
with the upper triangular $3\times3$ matrices over $\mathbb{F}_p$
with 1 on the diagonal).
Consider the amalgamated product
$\Gamma=E'\ast_{\IZ/p}E$ given by $\IZ/p=Z(E)$ and an injective map
$f:\IZ/p\rightarrow E$ with $f(\IZ/p)\cap Z(E')=1$.
Clearly $rk(\Gamma)=2$.
Let $\Gamma$ act on a finite dimensional CW-complex $X\simeq S^n$.
Consider the restriction of this action to $E$ and $E'$.
It is well known that the dimension function of a p-group action on a sphere is
realized by a representation over the real numbers \cite{dotzel}.
Therefore, an even multiple of the dimension functions
for $E$ and $E'$ must be
realized by characters $\chi_E$ and $\chi_{E'}$.

Clearly the dimension functions of $\chi_E$ and $\chi_{E'}$ must agree
over $Z(E)$ and $f(\IZ/p)$. Looking at the character table of $E$, we observe that every
irreducible character $\alpha$, giving rise to an effective sphere, vanishes
outside $Z(E)$ while $\alpha (z)=m\zeta_p$ for all $z\in Z(E)\setminus\left\{0\right\}$
(here $\zeta_p$ is a $p$-root of the unity). Thus, $\chi_E$ and $\chi_{E'}$ cannot be both
characters giving rise to effective spheres.
We deduce that the original action must
have some finite isotropy subgroups of rank 2. This provides
an example of an infinite group, with rank 2 finite p-subgroups,
not acting with effective Euler class on any sphere.


\subsection{Algebraic spheres.}

Let $\Gamma$ be a group with $vcd(\Gamma)<\infty$ and rank $r$.
As announced in the introduction of section 5, we want to show
that there is a finite dimensional, contractible $\IZ\left[\Gamma \right]$-complex
$\textbf{C}$
and $rk(\Gamma)$ finite dimensional
$\IZ \left[\Gamma\right]$-complexes $\textbf{C}_1,...,\textbf{C}_{r}$
such that $\textbf{C}=\textbf{C}_1\otimes...\otimes\textbf{C}_{r}$
is a complex of projective $\IZ \left[\Gamma\right]$-modules with
$H^{\ast}(\textbf{C})\cong
    H^{\ast}(S^{n_1}\times ...\times S^{n_{r}})$. We begin by recalling
some preliminaries concerning the cohomology of finite groups. We follow
here \cite{ben3} and \cite{ben2}. Let $G$ be a finite group. Consider
$\zeta\in H^n(G,R)\cong Ext^n_{RG}(R,R)\cong Hom_{RG}(\hat{\Omega}^nR,R)$,
where $\hat{\Omega}^nR$ is the $n$th kernel
in a $RG$-projective resolution $\bold{P}$ of $R$.
We choose a cocycle
$\hat{\zeta}:\hat{\Omega}^nR\rightarrow R$ representing $\zeta$.
By making $\bold{P}$ large enough we can assume that
$\hat{\zeta}$ is surjective. \nopagebreak We denote $L_{\zeta}$ its kernel
and form the pushout diagram: \nopagebreak
$$\xymatrix{
     & L_{\zeta} \ar[r]^{=} \ar[d] & L_{\zeta} \ar[d] & & & & & \\
    0 \ar[r] & \hat{\Omega}^nR \ar[r] \ar[d]^{\hat{\zeta}} & P_{n-1} \ar[r] \ar[d]&
        P_{n-2} \ar[r] \ar[d] & ... \ar[r] & P_0 \ar[r] \ar[d] & R \ar[r] \ar[d]^{=} & 0 \\
    0 \ar[r] & R \ar[r] \ar[d] & P_{n-1}/L_{\zeta} \ar[r] \ar[d]&
        P_{n-2} \ar[r]  & ... \ar[r] & P_0 \ar[r]  & R \ar[r]  & 0 \\
    & 0  & 0  & & & & &}$$
We denote by $\textbf{C}_{\zeta}$ the chain complex:
$$0 \rightarrow P_{n-1}/L_{\zeta} \rightarrow
    P_{n-2} \rightarrow ... \rightarrow P_0 \rightarrow R \rightarrow 0$$
formed by truncating the bottom row of this diagram. Thus we have
that $H_0(\textbf{C}_{\zeta})=H_{n-1}(\textbf{C}_{\zeta})=R$ while
$H_i(\textbf{C}_{\zeta})=0$ if $i\neq 0, n-1$.
A useful result is given in the proof of theorem 3.1 in
\cite{adem2}:

\begin{prop} \label{strumpf}
Let $G$ be a finite group. For all positive integer $r$, there exist classes
$\xi_1,...,\xi_r\in H^{\ast}(G,\IZ)$ such that, for all $H<G$ with $rk(H)\leq r$,
the complex
$\IZ\left[G/H\right] \otimes L_{\xi_1}\otimes ...\otimes L_{\xi_r}$
is $\IZ \left[G\right]$-projective.
\end{prop}

\begin{proof}
See the proof of theorem 3.1 in
\cite{adem2}
\end{proof}

\begin{cor} \label{strumpf2}
Let $G$ be a finite group. For all positive integer $r$, there exist
$r$ finite dimensional $\IZ\left[G\right]$-complexes
$\textbf{C}_1,...,\textbf{C}_r$ such that
$H^{\ast}(\textbf{C}_1\otimes...\otimes\textbf{C}_r)=H^{\ast}(S^{n_1}\times
...\times S^{n_r})$; with
$\textbf{C}_1\otimes...\otimes\textbf{C}_r$ a complex of
$\IZ\left[H\right]$-projective modules for all $H<G$ with $rk(H)\leq
r$.
\end{cor}

\begin{proof}
Let $\xi_1,...,\xi_r\in H^{\ast}(G,\IZ)$ be the classes given in proposition \ref{strumpf}.
Consider the chain complex $\textbf{C}=\textbf{C}_{\xi_1}\otimes ...\otimes \textbf{C}_{\xi_r}$.
Clearly $H^{\ast}(\textbf{C})=H^{\ast}(S^{n_1}\times ...\times S^{n_r})$.
For the second part of the claim, observe that
all the modules in $\textbf{C}_{\xi_i}$ are $\IZ\left[G\right]$-projective
except the module $P_{n_i-1}/L_{\xi_i}$. Recall that the tensor product
of any module with a projective module is projective, so that it remains to examine the module
$P_{n_1-1}/L_{\xi_1}\otimes ... \otimes P_{n_r-1}/L_{\xi_r}$.
Let $H<G$ be such that $rk(H)\leq r$.
Since $\IZ\left[G/H\right] \otimes L_{\xi_1}\otimes ...\otimes L_{\xi_r}$
is $\IZ \left[G\right]$-projective by proposition \ref{strumpf}, we conclude
that $\IZ\left[G/H\right]\otimes P_{n_1-1}/L_{\xi_1}\otimes ... \otimes P_{n_r-1}/L_{\xi_r}$
is $\IZ \left[G\right]$-projective as in 5.14.2 of \cite{ben2}. It then easily
follows that $P_{n_1-1}/L_{\xi_1}\otimes ... \otimes P_{n_r-1}/L_{\xi_r}$
is $\IZ \left[H\right]$-projective.
\end{proof}

We go back now to our group $\Gamma$ with $vcd(\Gamma)<\infty$ and rank $r$.
Write $\Gamma '$ for a torsion-free
normal subgroup of $\Gamma$ with $G=\Gamma/\Gamma '$ finite.
Write also $\mathfrak{E}\Gamma$ for a
contractible finite dimensional proper $\Gamma$-CW-complex.
Here proper means that the
isotropy subgroups are finite.
We apply corollary \ref{strumpf2} to $\Gamma / \Gamma'$
with $r=rk(\Gamma )$. We recover
a $\IZ\left[\Gamma \right]$-complex $\textbf{C}=
    \textbf{C}_{1}\otimes ...\otimes \textbf{C}_{r}$
such that
$H^{\ast}(\textbf{C}_1\otimes...\otimes\textbf{C}_r)=H^{\ast}(S^{n_1}\times ...\times S^{n_r})$;
with $\textbf{C}_1\otimes...\otimes\textbf{C}_r$ a complex of $\IZ\left[H\right]$-projective modules
for all finite $H<\Gamma$.
On the other hand, we have the contractible $\IZ\Gamma $-complex $C_{\ast}(\mathfrak{E}\Gamma)$.
Therefore the complex $\textbf{D}=C_{\ast}(\mathfrak{E}\Gamma)\otimes \textbf{C}$
is such that $H^{\ast}(\textbf{D})=H^{\ast}(S^{n_1}\times ...\times S^{n_r})$.

\begin{lem}
With the notation above, the complex $\textbf{D}$ is $\IZ\left[\Gamma\right]$-projective.
\end{lem}

\begin{proof}
The complex $C_{\ast}(\mathfrak{E}\Gamma)$ decomposes as a
direct sum of permutation modules:
$C_{\ast}(\mathfrak{E}\Gamma)=\oplus_{\sigma}\IZ\left[\Gamma /\Gamma_{\sigma}\right]=
    \oplus_{\sigma}\IZ\left[\Gamma \right]\otimes_{\IZ\left[ \Gamma_{\sigma}\right]}\IZ$.
Here $\sigma$ spans the cells of $\mathfrak{E}\Gamma /\Gamma$.
Consequently $\textbf{D}=\oplus_{\sigma}(\IZ\left[\Gamma \right]\otimes_{\IZ\left[ \Gamma_{\sigma}\right]}
     \textbf{C}_1\otimes...\otimes\textbf{C}_r)$, so that
we only need to prove that $\IZ\left[\Gamma \right]\otimes_{\IZ\left[ \Gamma_{\sigma}\right]}
     \textbf{C}_1\otimes...\otimes\textbf{C}_r$
is $\IZ\left[\Gamma \right]$-projective. Let
$Q_{\sigma}$ be a $\IZ\left[\Gamma_{\sigma} \right]$-module such that
$(\textbf{C}_1\otimes...\otimes\textbf{C}_r)\oplus Q_{\sigma}$ is
$\IZ\left[\Gamma_{\sigma} \right]$-free. We then have that
$(\IZ\left[\Gamma \right]\otimes_{\IZ\left[ \Gamma_{\sigma}\right]}
     \textbf{C}_1\otimes...\otimes\textbf{C}_r)\oplus
    (\IZ\left[\Gamma\right]\otimes_{\IZ\left[ \Gamma_{\sigma}\right]}Q_{\sigma})=
    \IZ\left[\Gamma \right]\otimes_{\IZ\left[ \Gamma_{\sigma}\right]}
    ((\textbf{C}_1\otimes...\otimes\textbf{C}_r)\oplus Q_{\sigma})$
is $\IZ\left[\Gamma\right]$-free.

\end{proof}

\newpage
We summarize the main result of this subsection in the following:

\begin{cor}
For a group $\Gamma$ with $vcd(\Gamma)<\infty$ and $rk(\Gamma)=r$,
there exist a finite dimensional contractible complex $\textbf{C}$
and $r$ finite dimensional $\IZ\left[\Gamma\right]$-complexes
$\textbf{C}_{1}$,...,$\textbf{C}_{r}$ such that
$\textbf{D}=\textbf{C}\otimes \textbf{C}_{1}\otimes ...\otimes \textbf{C}_{r}$
is a $\IZ\left[\Gamma\right]$-projective complex with
$H^{\ast}(\textbf{D})\cong
    H^{\ast}(S^{n_1}\times ...\times S^{n_{r}})$.
\end{cor}

\end{document}